\documentclass{amsart}
\usepackage{amscd, amsfonts, amsmath, amsthm, amssymb}
\usepackage{tabularx}
\usepackage{longtable}
\usepackage{hhline}
\usepackage{enumerate}
\usepackage{booktabs}
\usepackage{comment}
\usepackage
[  backref,
letterpaper = true,
hyperindex = true,
linktocpage = true
]
{hyperref}

\newcommand{\Z}{\mathbb{Z}}
\newcommand{\F}{\mathbb{F}}

\theoremstyle{plain}
\newtheorem{theorem}{Theorem}[section]
\newtheorem{conjecture}[theorem]{Conjecture}

\theoremstyle{definition}

\newtheorem{remarks}[theorem]{Remarks}
\newtheorem*{remark*}{Remark}
\newtheorem*{remarks*}{Remarks}

\numberwithin{table}{theorem}

\begin{document}

\author{Jos\'e Alejandro Lara Rodr\'iguez and Dinesh S. Thakur}

\title[Zeta-like Multizeta Values]{Zeta-like 
Multizeta Values for higher genus curves}

\address{Department of Mathematics, University of Rochester, 
Rochester, NY 14627 USA, dinesh.thakur@rochester.edu}

\address{
Facultad de Matem\'aticas, Universidad Aut\'onoma de Yucat\'an, Perif\'erico
Norte, Tab. 13615,
M\'erida, Yucat\'an, M\'exico, alex.lara@correo.uady.mx}




\begin{abstract}
We prove or conjecture several relations between the multizeta 
values for positive genus function fields of class number one, focusing on the zeta-like values, namely those 
whose ratio with the zeta value of the same weight is rational 
(or conjecturally equivalently algebraic). These are the first known relations between 
multizetas, which are not with prime field coefficients. We seem to have one universal family. 
We also find that interestingly the mechanism with which the relations work 
is quite different from the rational function field case, raising interesting questions 
about the expected motivic interpretation in higher genus.   We provide some data in support 
of the guesses. 
\end{abstract}

\maketitle

\section{Introduction}

Recently studied connections of the multizeta values, introduced by Euler,  with the arithmetic 
fundamental groups have made them an  important tool in the recent push towards non-abelian, 
homotopical directions in number theory. See e.g., \cite{Z16} and references there to the huge body of work by several mathematicians.  

For a survey of work on the function field analog, with connections to Drinfeld 
modules and Anderson's t-motives etc. (see \cite{A86, G96, T04} for background), we refer to the survey \cite{T17}.

Let us focus on very simple type of relations between the multizeta values. 
Following \cite{LT14}, we call a multizeta value 
zetalike if its ratio with the zeta value of the same weight is rational. In case of 
even weight, we also call it Eulerian. (Often we restrict to multizeta of depth more than one, 
without mention, since only then the concept is really significant). 
In the case of rational number field, we know some
eulerian families \cite{LT14}, but we speculated that only $\zeta(2m+1)$ and its duals 
may be only multizetas that are  zetalike of odd weight. In contrast, we proved  \cite{LT14} (see also \cite{CPY19,  C17, T17}) (conjecturally) all the Eulerian multizetas for the rational function field case, but could only prove and conjecture several zetalike families of odd weight, without getting full characterization, it seems. 

In this paper, we investigate this question for higher genus function fields, where so far 
the only relations known \cite{T10} were the sum shuffle type relations (with prime field coefficients) 
(excluding the obvious relations $\zeta(ps_1, \cdots, ps_r)=\zeta(s_1, \cdots, s_r)^p$ in characteristic $p$.)
The t-motivic period interpretation \cite{AT09} in depth more than one is also only developed so far in genus zero. 
Now we find (and can prove some) much more interesting relations involving 
non-prime field coefficients. 

Several years ago, the second author had checked (numerically) that $\zeta(1, q-1)$ is 
not zetalike  (for one class number one elliptic curve over $\F_q$, with $q=2$), in contrast to what he had 
proved \cite{T04, T09} in the rational function field case over $\F_q$. That this multizeta being zetalike is (conjecturally) the only non-trivial linear relation in weights at most $q$. So in contrast to the rational function 
fields, in higher genus, it seems that the relations start at higher weights. 

Now with  more extensive use 
of computer aided numerical experiments, we have better understanding (see below) of what should 
happen and some `positive identifications' of zetalike multizeta. 

We then  prove 
some of these conjectures by developing further, from the zeta case to multizeta, the `polylog-algebraicity techniques' of \cite{T92}, where an appropriately constructed  algebraic function on the curve cross itself (or Hilbert cover cross 
itself in general) specialized at graph of the $d$ th power of Frobenius gives appropriate power sums of degree $d$ (or at most $d$), for all $d$.  In \cite{T92}, these special algebraic functions (called F-functions) were used to give motivic algebraic incarnation
of some zeta values (especially at $1$, in class number one situation) generalizing partially the results of \cite{AT90} to higher genus. 
In 2009, these were used to verify \cite{Tp} that Taelman' s beautiful analog of  
analytic class number 
formula \cite{Ta10}, which was then made only for the base $\F_q[t]$, works also for the
 higher genus cases of class number one.  Various  aspects of log-algebraicity were 
 developed much further in e.g., \cite{A94, A96, F15, APT16, D16, GP18, M18} by  Anderson, Angl\`es, Derby, Fang, Green, Mornev, Ngo Dac, Papanikolas, Pellarin, Taelman, Tavares-Ribeiro,  and \cite{T04}[Sec. 8.9-8.10]. 
We need to extend these techniques to adapt to multizeta. We have not resolved this issue of the extent of log-algebraicity fully
in this paper, but just sufficient to prove our theorems and to illustrate the techniques. 

Interestingly, the proofs as well as the mechanisms how these identities work 
out at infinite level, as  limits from finite levels,  are now quite different than in the genus zero case. 
 It shows that we will need a 
better understanding of the underlying t-motivic mechanisms to understand the situation in general. (See Remark (II) in Section 4.) In informal terms, the motives here are constructed through such functions and the motivic identities are identities between such functions, and the iterated 
sum mechanism comes via the Frobenius-difference matrix equations \cite{AT09}[Sec. 2.5] 
that arise in Anderson's theory of $t$-motives. We expect (and know to a large extent 
by the works of Anderson, Brownawell, Papanikolas, Chang and Yu) that the multizeta values relations 
come from motivic identities, but in our higher genus, in contrast to what we know in genus zero, 
though it is certainly not ruled out, we have not been able to do this (as explained in Section 4), 
but have resorted to different mechanism of proofs, which raise some interesting questions and formulations.

We find (with only numerical evidence in low weights)  exactly one zetalike/
eulerian (`primitive', i.e., tuple not a multiple of $p$) family $\zeta(q^n-1, (q-1)q^n, \cdots, (q-1)q^{n+k})$, where $\F_q$ is the 
field of constants, surviving  from the rational function field case,  
for all (4 of them) class number one situations of higher genus. 
We have not yet found any zetalike example in odd weights in the higher genus case.

We did not find any zetalikes in higher class number ones, and speculate that probably 
there are families of Hilbert class field coefficient linear combinations of multizeta 
values of different ideal classes  (for the same tuple of $s_i$'s) that are algebraic multiples of zeta of the same weight, but it might be rare or impossible for a single value to be zetalike in this case, unless 
you take all ideal classes into account. 

In the function field analog that we investigate (see \cite{T17} for survey and references), 
the relations 
are still not conjecturally well-understood, though in contrast, there 
are also some very strong  transcendence and linear/algebraic independence 
results (by Anderson, Brownawell, Chang, Papanikolas, Yu, et al)  proved. Note that the various transcendence, independence results that have been proved for the zeta immediately 
carry over to the zetalike multizeta.

We  first fix the notation and give the basic definitions.  
Next, we  state our  conjectures on the zetalike families and  give the proof 
of the special cases of conjectures. Then we give several remarks on possible generalization 
of the proof techniques, the contrasts 
with genus zero case and motivic implications. Then we discuss  the relative situation.
 Finally,  we discuss the  numerical data, calculated by the first author,
  giving some evidence for the  conjectures made from it.

  \section{Notation and definitions}
  
 Consider a  function field $K$ (of one variable over finite field $\F_q$), having 
a  rational, or equivalently of degree one place. We choose any such place and 
label it $\infty$.  Denote the corresponding ring of integers by $A$ (consisting 
of elements of $K$ with only pole at $\infty$, the completion by $K_{\infty}$, 
 and the completion  of its  fixed algebraic closure by $C_{\infty}$. Fix 
 an uniformizer at infinity, so that we have corresponding sign (and degree) function.
 Let $A+$ ($A_d+$, respectively) denote the set of monic, i.e., of  sign 1 (monic  of degree $d$ respectively) 
 elements of $A$.

  For $k, k_i, d\in \Z$, consider the power sums (sometimes denoted by $S_d(-k)$ in the references) 
$$S_d(k)=\sum_{a\in A_d+}\frac{1}{a^k}\in K, $$
and extend inductively to the iterated power sums 
\begin{align*}S_d(k_1, \cdots, k_r) & = & &S_d(k_1)& &S_{<d}(k_2, \cdots, k_r)&\\
& = &  &  S_d(k_1) & &\sum_{d>d_2>\cdots>d_r}
S_{d_2}(k_2)\cdots S_{d_r}(k_r),& \end{align*}
where $S_{<d}=\sum_{i=0}^{d-1}S_i$ as the notation suggests.

For positive integers $s_i$, we consider the multizeta values
$$\zeta(s_1, \cdots, s_r):= \sum_{d=0}^{\infty} S_d(s_1, \cdots, s_r)=\sum  \frac{1}{a_1^{s_1}\cdots a_r^{s_r}}\in K_{\infty},$$
of weight $\sum s_i$ and depth $r$ (associated a priori to the tuple $s_i$ rather than the value). 
(Here the second sum is over monic $a_i\in A$ of strictly decreasing degrees). 

Call $\zeta(s_1, \cdots, s_r)$ {\it zetalike} (we only care, if 
$r>1$) if $\zeta(s_1, \cdots, s_r)/\zeta(\sum s_i)
\in K$. 

In the case the weight $\sum s_i$ is $q$-even (i.e., a multiple of $q-1$), we also 
call the zetalike value {\it eulerian}, in recognition of the simple evaluation by Euler in the rational 
case, and analogous evaluations \cite{C35, T04} by Carlitz and Goss in function fields. 

Finally,  for $\rho$ a  sign normalized rank one Drinfeld $A$-module (also called Hayes module), we denote the  corresponding 
exponential and logarithm functions as $\exp_{\rho}(z)=\sum z^{q^i}/d_i$ and $\log_{\rho}(z)=\sum z^{q^i}/\ell_i$. While $\ell_i$ and $d_i$ are polynomials in $t$ in the $A=\F_q[t]$ case, in higher genus case, they are rational functions (non-integral in general) in the Hilbert class field.
(see e.g., \cite{T04}, Chapter 2 for details).

 \section{Class number one situation: Conjectures and theorems} 
  
 Apart from $A=\F_q[t]$'s  (one for each prime power $q$), there are exactly 
 four (see \cite{T04} for references and corresponding Hayes modules) other $A$'s of class number one: 
 
 \begin{itemize} 
 
 \item
 (i)  $\F_2[x, y]/y^2+y=x^3+x+1$, 
  
  \item
   
 (ii) $\F_3[x, y]/y^2=x^3-x-1$, 
 
 \item
 
 (iii) $\F_4[x, y]/y^2+y=x^3+w$, where 
 $w^2+w+1=0$, 
 
 \item 
 
 (iv) $\F_2[x, y]/y^2+y=x^5+x^3+1$.
 
 \end{itemize} 
 
  Note that the first three are of genus 1 while the last one is of genus 2.

\begin{conjecture} For any class number one $A$ with constant field $\F_q$, the multizeta values 
$\zeta(q^n-1, (q-1)q^n, \cdots, (q-1)q^{n+k})$ are zetalike. 
\end{conjecture}

\begin{remarks} (i) For the case of $A=\F_q[t]$'s, following  more explicit form below was conjectured, proved in depth 2  in \cite{LT14} and proved for any depth by Chen in \cite{C17}. 

\begin{align*}
\zeta(q^n-1, (q-1)q ^n, \dotsc, (q-1)q^{n+k})
=
\frac{[n+k] [n+k-1] \dotsm [n]}
{[1]^{q^{n+k}} [2]^{q^{n+k-1}} \dotsm [k+1]^{q^n}  }
\zeta(q^{n+k+1}-1),
\end{align*}
where $[n]=t^{q^n}-t$.

(ii) In genus zero case, there are more such families \cite{LT14}, but in higher genus, our limited 
exploration leads only to the family in the conjecture.   (Of course, we restrict to `primitive'  tuples, 
i.e.  not divisible by  the characteristic). 
\end{remarks} 

 Here are some more explicit conjectures in higher genus, class number one cases, listed above.

\begin{conjecture}  Put $R_n=\zeta( q^n-1, q^n (q-1)) / \zeta ( q^{n+1}-1)$.

For the case (i), we have 
$$R_n= \frac{x^{2^{n+1}}+x^2}{y^{2^{n+1}}+y+x^{2^{n+1}+1}+x}.$$

For the case (ii), we have 
$$R_n=\frac{ (x^{3^n}-x)(y^{3^n}-y)^2 + (x^{3^n}-x) ( -x^{3^n} - x^3-x+1)}{x^{2 * 3^{n+1}} +  x^{1+3^{n+1}} + x^{3^{n+1}} +  y^{1+3^{n+1}} +x^2-x+1}.$$

For the case (iii), we have 
$$R_n=\frac{(x^{4^n}+x)(y^{4^{n+1}}+y^4) + (x^{4^{n+1} }+x^4) \left( x^{4^n +2} +x^3+1 \right) +(x^{4^n}+x) }{ x^{2 \cdot 4^{n+1} +2}  + x^{4^{n+1}} y^{4^{n+1}} + x^{4^ {n+1}}y +x^{4^{n+1}}  +x y^{4^{n+1}}   + xy }.$$

For the case (iv), we have 
$$R_n=\frac{X^{22}+(1+x)(X^{20}+X^{18}+X^{16})+(1+x+x^2)(X^{12}+X^{10})+X^9+LR_n}{X^{24}+xX^{16}+(x+1)X^8+x^2+x},$$
where $X=x^{2^{n-1}}$ and $Y=y^{2^{n-1}}$ and $LR_n=xX^8+X^5
+(Y+y)(X^2+X^4)+x^2+x$.

\end{conjecture}

Note that the fractions in the conjecture are not in the reduced form, and in fact, there is a lot of cancellation (making it hard to guess from numerical data!). Compare for example the reduced forms 
in the special case of the theorems below. We have numerically verified the  case (i) for $n\leq 11$ , and (ii) for $n\leq 9$, (iii)  for $n\leq 5$ and (iv) for $n\leq 12$. 

We also have some more such explicit ratio conjectures, but not yet for large satisfactory families.

 Our main theorems below  prove the conjecture in higher genus, namely the  case (i), 
  when $n=1,2$,   $k=0$. (We have also proved  the case (ii, iii, iv) for $n=1$ with 
  the exact same method, but  choose to give all the  details in the appendix to the main paper).   We will use the theory of \cite{T92} (see also \cite{T04}[Sec. 4.15, 8.2]). 
  The equation references such as (14) in this section all refer to this paper \cite{T92}. 
  
  \begin{theorem}
  For $A=\F_2[x, y]/(y^2+y=x^3+x+1)$, we have 
  
  $$(x^2+x+1)\zeta(1, 2)=\zeta(3).$$
  
  \end{theorem}
  
 \begin{proof} 
 
  We first give some explicit `(poly-)log-algebraicity' formulae developed in \cite{T92}
  for the relevant power sums,  and in  \cite{T09} for the iterated versions, together with one extension 
  needed for the multizeta case. These give expressions for polylog-coefficient $\ell_d^w$ times
  power sums $S_d(k_1, \cdots, k_i)$'s in terms of algebraic functions, on the curve (corresponding to $K$) cross itself,  specialized at the graph of $d$-th power of Frobenius map. 
  
  We will now define several functions in $\F_2(x, y, X, Y)$, where $x$ and $X$ are independent transcendentals and $y^2+y=x^3+x+1, Y^2+Y=X^3+X+1$. For each such function, say $f$, 
  put $f^{(1)}$ for the function resulting from $f$ after substituting $X^2, Y^2$ respectively for 
  $X, Y$, and put $f(d)\in K$ for the function resulting from $f$ after substituting 
  $x^{2^d}, y^{2^d}$ for $X, Y$. 
 (More generally, we say \cite{T92}, in class number one case,  that a function $f: \Z_{>c}\rightarrow C_{\infty}$ is $F$-function, if there is a rational  function $F$ on $C$ cross itself  
such that $F$ specialized to the graph of $d$-th power of Frobenius on $C$ is $f(d+k)$ (for sufficiently large $d$, fixed $k$).)

  Put 
  
  $$B_x=X+x, \ \ \ B_y=Y+y,\ \ \ g=\frac{B_y+XB_x}{B_x^{(1)}+1},\ \ 
  F_1=\frac{X+x^2}{B_y+xB_x+x^2+x},$$
  
  $$ F_{<1}=g^2+F_1^2+F_1,\ \ \ F_{12}=F_1F_{<1}^2, \ \ \  F_3=F_1^2(g^2+F_1^2),$$
  
  $$g_m= \frac{Y^2+y^4+X^2(X^2+x^4)}{X^4+x^4+1},\  \ \ A_2=F_1(\frac{g^4g_m}{x^2+x}+(g^2+F_1^2+F_1)^2(F_1^2+F_1)),$$ 
  
  $$F_{<3}=\frac{A_2}{F_1}+(g^2+F_1^2+F_1)^3,$$
  
  $$C=B_y+xB_x+x^2+x,C_m=Y+y^2+x^2(X+x^2)+x^4+x^2,$$  
  
  $$U=\frac{(X+x^2)(X^3+X^2x)}{x^2+x}+
  \frac{X^2}{x}+Xx+1,\ \ \ J=\frac{U+C^2}{1+((g^{(1)})^3C^2F_{12})/C_mF_{12}^{(1)})},$$
  
  $$F_{\leq 12}=\frac{JF_{12}}{C_m}.$$

  \begin{theorem} For $A$ as in the previous theorem and for $\ell_d$ the 
  (reciprocal) coefficients of logarithm for Hayes module for this $A$,  for 
  $d\geq 2$ (check d=0, 1, 2) we have 
  (i) $\ell_dS_d(1)=F_1(d)$, (ii) $\ell_dS_{<d}(1)=F_{<1}(d)$, \\
  (iii) $\ell_d^3S_d(1, 2)=F_{12}(d)$, (iv) $\ell_d^3S_d(3)=F_3(d)$, \\
  (v) $\ell_d^3S_{<d}(3)=F_{<3}(d)$ and (vi) $\ell_d^3S_{\leq d}(1, 2)=F_{\leq 12}(d)$. 
  \end{theorem}
  
  \begin{proof}
 We use the generating functions $A_{d0}/(1-\sum A_{di}t^{q^i})=\sum S_d(k)t^{k-1}$ of 
 \cite{T92}[(18)] for $S_d(k)$ and $(A_{d0}x)(\sum A_{di}x^{q^i})^{-1}=
 1+\sum S_{<d}(k) x^k$ of \cite{T09}[3.2] for 
 $S_{<d}(k)$ given by one type of binomial 
 coefficient \cite{T92, T09}, and the method of \cite{T92} to calculate this in higher genus. 
 (Here $i$ runs from $0$ to $d$ and $k$ from $1$ to $\infty$, and $S_{<d}(k)=0$ unless 
 $k$ is $q$-even.)

  We first explain briefly, how (i)-(v) follow from theory developed in 
  \cite{T92}, by unwinding and specializing 
  the genus one formulas there to our specific $A$. Note the notations 
  matches  $B_x(i)=[i]_x, B_y(i)=[i]_y$. Once one uses the known coefficients of $\rho$  (see Exa. C page 192 
  of \cite{T92}) to get $x_1=x^2+x, y_1=y^2+y, y_2=x(y^2+y)$,  the recursions 
  for $\ell_i, d_i$ (and this $a_{ik}$ by (7), (14) of \cite{T92}) from the functional equations of logarithm, exponential in  terms of $\rho$, 
  give formulas (we use $\ell_1=1, d_1=1$ in particular) for $f_i, g_i, \mu_i$ in (20), (27), (28), (23) 
  of \cite{T92} implying in particular that  $g(i)=g_i=\ell_i/\ell_{i-1}$. This allows us to calculate 
  $A_{i0}, A_{i1}, A_{i2}$ of (14) by comparing $t, t^q, t^{q^2}$ coefficients in (21) (see (14), (7)),  which is all we need from the generating function coefficients. (In fact, $A_2(d)=\ell_d^4A_{d2}$ 
  and $(g^2F_1+F_1^3+F_1^2)(d)=\ell_d^2A_{d1}$.)
  Then we 
  need only to verify by straightforward manipulation that (we note here that $g_m(d)=g(d-1)^4$, 
 $C_m(d)=C(d-1)^2$)
  $$\ell_d S_d(1)=\ell_dA_{d0}=F_1(d),$$
  
  $$\ell_dS_{<d}(1)=\frac{\ell_d^2A_{d1}}{\ell_dA_{d0}},$$
  
  $$\ell_d^3S_d(3)=
  \ell_dA_{d0}(\ell_d^2A_{d1}+\ell_d^2A_{d0}^2)=F_3(d),$$
  
  $$\ell_d^3S_d(1, 2)=(\ell_dS_d(1))(\ell_dS_{<d}(1))^2=(\ell_dA_{d0})(\frac{\ell_d^2A_{d1}}{\ell_dA_{d0}} )^2=F_{12}(d),$$
  
  $$\ell_d^3S_{<d}(3)= \frac{\ell_d^4A_{d2}}{\ell_dA_{d0}}+(\frac{\ell_d^2A_{d1}}{\ell_dA_{d0}})^3             =F_{<3}(d).$$

  Finally, we verify (vi) by induction on $d$, using (ii) and the iterated definition: it is 
  enough to check the initial value and the identity  corresponding  to $S_{\leq d+1}-S_{d+1}=S_{\leq d}$.  Since 
  $g^{(1)}(d)=\ell_{d+1}/\ell_d$, the identity is 
  $F_{\leq 12}^{(1)}- F_{12}^{(1)} = (g^{(1)})^3F_{\leq 12}$, which follows directly. 
 
\end{proof}

Now it is easy to finish the proof of the main theorem by just noticing that 
$(x^2+x+1)F_{\leq 12}-F_{<3}$ has negative degree in $X, Y$, so that 
as $d$ tends to infinity, the `error' $(x^2+x+1)S_{\leq d}(1, 2)-S_{<d}(3)$ tends to 
zero, establishing the theorem. 
  \end{proof}

  \section{Some remarks} 
(I)   {\bf Explicit F-functions in standard forms:} To get more concrete perspective, we give some of these functions more explicitly. 
We split lower order part of numerators just for the display convenience.

  $$F_3 = \frac{X^3+x^{2} X^{2} + Y+X+x^{3}  + x^{2} + y+x  + 1}{X^{4} + x^{2} + 1}.$$

  $$F_{12} = \frac{X^4Y+xX^5+X^3Y+(y+1)X^4 +x^2X^2Y+ (x^3+y+x)X^3+(x+1)XY+L_{12}}{X^6 + (x+1) X^4 + (x^2+1) X^2 + x^3 + x^2+  x + 1}$$
with $L_{12}=(x^2y+x^3+x^2+x)   X^2 +x^3Y+(x^4+xy+x^2+y+x)X+x^3y+x^4+x^2$,

$$F_{<3} = \frac{(x^2+x+1)X^6+X^5+(x^4+  x^3+x^2+x)X^4 +X^2Y+(x+1)X^3 + L_{<3} }{(x^2+x)X^{6} +(x^3+x)X^4+(x^4+x^3+x^2+x) X^2+x^5+x}$$
with $L_{<3}=(x^4+x^3+x^2+y)X^2+xY+xX+x^6 +x^5 +x^3 +xy+x^2 $

$$F_{\leq 12}=\frac{X^6+(x^2+x+1)X^5+(x^3+x^2) X^4+(x^2+x+1)X^2Y+ L_{\leq 12} }{(x^2+x) X^6 +(x^3+x)X^4+(x^4+x^3+x^2+x)X^2+ x^5+x}$$
with $L_{\leq 12}=  (x^3+x^2+x+1)X^3 +(x^2y+x^3+ xy+y)X^2+x^3Y+x^3X+x^5 +x^3 y+x^3$.

Note that the denominators are $(X^2+x+1)^2$, $(X^2+x+1)^3$, $(x^2+x)(X^2+x+1)^3$ (twice) 
respectively. 

Comparing the dominating terms of the last two expressions, makes visible the last calculation 
of the proof above. 
  
 (II) {\bf Comparison with the genus zero situation:} For the genus zero case $A=\F_q[x]$, we have \cite{T09, T04} the $F$-function identity 
 $S_d(q-1, q(q-1))=S_{d-1}(q^2-1)/(t-t^q)^{q-1}$, which by summing over degrees up to $d$ then gives 
 corresponding $F$-function identity for $S_{\leq d}$,  and then,  by taking the limits, the  identity at the 
 multizeta  level. The same is true in any depth by the formula for $S_n(d)$ in the proof on page 795 in 
 \cite{LT14}. On the other hand, in our case here, we have the identity only 
 at the level of $\zeta$, only leading terms matches at $S_{\leq d}$ level, and nothing 
 at $S_d$ level! In fact, for $q=2$ case above
 (for example, by the theorem 2 and (I)), for $d>2$, the degree of both $S_d(1, 2)$ and $S_{d-1}(3)$ is 
 $-2^{d}$.   The  degree of $S_d(1, 2)+S_{d-1}(3)+(x^2+x)S_{d+1}(1, 2)$ is 
 $-2^{d+1}$, and the degree of $(x^2+x+1)S_{\leq d}(1, 2)+S_{<d}(3)+(x^2+x)S_{d+1}(1, 2)$ 
 is $-2^{d+2}$. 
 
 Consider the genus zero zetalike Euler basic identity $\zeta(1, q-1)=\zeta(q)/(t-t^q)$, 
 which is not eulerian,  if $q>2$. It corresponds to $F$-function identity at $S_d$  level, we do not think that the resulting 
 identity at $S_{\leq d}$ level obtained by summing is $F$-function identity. 
 
We have checked (by computing the rank of the relevant matrices) that in the case (i), there is no non-trivial linear relation (leading to our theorem by summing) between the $8$ quantities 
$S_k(3), S_k(1, 2), S_k(2, 1), S_k(1, 1,1)$'s,  with $k=d$ or $d+1$, of weight $3$, 
(working for all $d$, or equivalently working for the corresponding F-functions), in contrast to the existence of such `binary' relations \cite{To18}.  We have also checked that the same situation persists, even if we add 
$S_{d+2}(1, 2), S_{d-1}(3), S_{d+2}(3)$, but have not tried adding more terms. 
Similarly, we have checked that there is no non-trivial linear relation (leading to our theorem by summing) between the $10$ quantities $S_k(1, 2), S_m(3)$, with $d\leq m\leq d+4, d<k\leq d+5$, 
again in contrast to the genus zero situation \cite{To18}. (See also \cite{T17}[Pa. 17-18] for relevant discussion 
of Todd's data.)

Since the $F$-functions involved in $S_d$-level identities were used to define \cite{AT90, AT09} 
the 
corresponding motives in the genus zero case, we need to understand better the motivic mechanisms underlying these relations in higher genus.

(III)  {\bf Comparison of polylog-algebraicity for iterated sums and zetalike property:}  As  mentioned above,  for a positive integer $k$ and  a positive $q$-even integer $m$, 
$\ell_d^kS_d(k)$ and $\ell_d^mS_{<d}(m)$ are $F$-functions \cite{T92, A94, T09, GP18} (we say alternately that 
$S_d(k)$ and $S_{<d}(m)$ satisfy `log-algebraicity' property). 

In our situation, if weight $w=\sum s_i$ is $q$-even, and if $\ell_d^wS_{<d}(s_1, \cdots, s_r)$
is $F$-function, then as in the proof above, comparison of leading terms shows that 
$\zeta(s_1, \cdots, s_r)$ is zetalike (equivalently, eulerian, in this case). 

 For simplicity, let $s_1, s_2$ be $q$-even, $w=s_1+s_2$. If $F_{\leq}$ is the $F$-function 
for $\ell_d^wS_{\leq d}(s_1, s_2)$, and if $F$ is the $F$-function representing 
$\ell_d^wS_d(s_1, s_2)$, and $g$ represents the $F$-function $\ell_d/\ell_{d-1}$, 
then $F_{\leq}(d)-F(d)=g(d)^{w}F_{\leq}(d-1)$, so to get such $F_{\leq}$ from 
(known) $F$ and $g$, we need to solve the Frobenius-difference equation 
$F_{\leq}-g^{w}F_{\leq}^{(-1)}=F$. When exactly is it solvable? If this is understood, the method 
explained in the next section to solve it
 should  then give (case-by-case) proofs for such multizeta relations through directly verifiable relations 
between such functions. 

(IV) {\bf Origin/explanation of some functions introduced in the proof:} For interested reader, we indicate how the formula (vi) was discovered (without 
knowledge of such algorithm). To guess what $F_{\leq 12}$ should be,  comparison 
of the LHS of (vi) with $F_{12}(d)$ was made (for small $d$'s) and factored to notice 
match of denominators, so their ratio was considered.  (Note that without the factor  
$\ell(d)^3$, the relevant denominators do not match!).  Again consideration of factors suggested 
that primes in denominators came from those of $F_{12}(d-1)$. Now the expressions  show that 
cube of $C(d)=[d]_y+x[d]_x+[1]$ kills denominator of $F_{12}(d)$, here the square was enough 
so  polynomials $E(d)=C(d-1)^2\ell_d^3S_{\leq d}(1, 2)/F_{12}(d)$ were calculated for a few $d$'s  
and it was noticed that except for the constant term which alternated between $0$ and $1$, 
the tail of $E(d)$ matched with $E(d-1)$, so the recursion between $E(d)$ and $E(d-1)$ 
was considered as F-polynomials are easy to guess explicitly (using geometric series). 
This led to function $U(d)$ satisfying $E(d)=U(d-1)+E(d-1)$.  
Next we consider equation coming from the relation  $S_{\leq d}=S_d+S_{\leq d-1}$, 
which, after a simple straight manipulation, translates to 
$E(d)=C(d-1)^2+E(d-1) g(d)^3C(d-1)^2F_{12}(d-1)/(C(d-2)^2F_{12}(d))$.
Solving these two equations  led to 
F-function expression for $E$ and thus for $F_{\leq 12}$. For more streamlined version developed later, see the next section.

  (V) {\bf Structure behind the explicit conjecture:} In the notation of the theorem, the first depth 2 explicit conjecture is  $\zeta(q^n-1, (q-1)q^n)/\zeta(q^{n+1}-1)=[n]^2/C(n+1)$. We have similar but more involved descriptions for the rest.
  We remark that the denominators listed above in each case are (Frobenius twists of) denominators of F-function $F_1$ satisfying $\ell_dS_d(1)=F_1(d)$.
  
(VI) {\bf Low $F_i$'s:}  If $F_k$ denotes the F-function with $F_k(d)=\ell_d^kS_d(k)$, then 
for $F_{mp^n}=F_1^{mp^n}$, for $m\leq q$, by $p$-th powers and $\F_q$-linearity 
and power sums-symmetric sums  argument \cite{T14}[Remark 6.1].

  \section{Another case}
  
  In order not to make the theorem and proof any more complicated by combining too many formulae at once, we decided to state the 
  second case separately as the following theorem.

  \begin{theorem}
  For $A=\F_2[x, y]/(y^2+y=x^3+x+1)$, we have 
  
  $$(x^8+x^6+x^5+x^3+1)\zeta(3, 4)=(x^4+x^2)\zeta(7).$$
  \end{theorem}
  
 \begin{proof} 
 We proceed in a similar way to the proof of the first case. In fact, the functions 
 interpolating $\ell_d^3S_d(3), \ell_d^3S_{<d}(3)$, are already computed and 
 since $S_d(2^ns)=S_d(s)^{2^n}$, the similar interpolating functions for $s=4$ 
 are just fourth powers of the functions we calculated above for $s=1$. 
 
 This gives $F_{34}$ such that $\ell_d^7S_d(3, 4)=F_{34}(d)$. Here it is explicitly:
 $F_{34}=N_{34}/D_{34}$, where $D_{34} =(X^2+x+1)^6$ and 
 \begin{align*}
N_{34}& =  X^{11} + x^{2} X^{10}  + \left(x^{3} + x + 1+y\right) X^{8} +
\left(x^{4} + 1\right) X^{7} +
 \left(x^{6} + x^{3} + x^{2} + x + y\right) X^{6} \\
&\quad  +
 \left(x^{4} + x^{2} + 1\right) X^{5} +\left(x^{7} + x^{6} + x^{5} + x^{4} + x^{2}+yx^4 \right) X^{4} 
 + 
 \left(x^{6} + x^{2} + 1\right) X^{3}\\
 &\quad  + 
 \left(x^{8} + x^{5} + x^{4} +  x^{2} + x + y(x^2+1)\right) X^{2}+ x^{6} X  + x^{9} + x^{8} + x^{7}  + x^{4} +yx^6
 \\
 &\quad  +Y\left[ X^{8} + X^{6} + x^{4} X^{4} + \left(x^{2} + 1\right) X^{2} + x^{6}      \right].
 \end{align*}

 We claim that  $F_{\leq 34}=N_{\leq 34}/D_{\leq 34}$ satisfies $\ell_d^7S_{\leq d}(3, 4)=F_{\leq 34}(d)$, where $D_{\leq 34}=(x^6+x^5+x^4+x^3+x^2+x+1)(X^2+x+1)^7$ and $N_{\leq 34}$ is
 \begin{align*}
&\left(x^{2} + x\right) X^{14} + 
\left(x^{2} + x\right) X^{13} + 
\left(x^{6} + x^{5} + x^{4} + x^{3} + x^{2} + x + 1\right) X^{12} \\ 
& +
\left(x^{6} + x^{5} + x^{4} + x^{2} + 1\right) X^{11} +
\left(x^{8} + x^{7} + x^{6} + x^{5} + x^{4}  + x^{2} + x+y(x^2+x)\right) X^{10} \\ 
& + \left(x^{8} + x^{4} + x^{3} + x^{2}\right) X^{9}
 +
\left(x^{8} + x^{7} + x^{6} + x^{5}  + x^{3} + y(x^6+x^5+x^4+x+1)\right) X^{8} \\ 
& +
\left(x^{9} + x^{7} + x^{6} + x^{4} + x^{3}\right) X^{7} \\
& +
\left(x^{9}+ x^{7}  + x^{5} + x^{2}  + x +1+ 
y (x^8+x^6+x^5+x^3+x^2+x + 1)\right) X^{6}  \\
& +
\left(x^{10} + x^{9} + x^{8} + x^{4} + x^{2} + x + 1\right) X^{5}\\
& +
\left( x^{7} + x^{6}  + x^{5}  + x^{3} +x+1+
+y(x^9+x^8+x^7+x^4+ x^2+ x + 1)\right) X^{4} \\
& +
\left(x^{11} + x^{10} + x^{9} + x^{7} + x^{5} + x^{4}\right) X^{3} \\
& +
\left(x^{12}  + x^{10} + x^{9} + x^{8} + x^{6} + x^{5} + x^{4}+y(x^{10}+x^7+x^5)\right) X^{2} \\
& +
\left(x^{11} + x^{9} + x^{4}\right) X + x^{10} + x^{8} + x^{7} + x^{6} + x^{5}  + x^{4}+y(x^{11}+x^9+x^4)\\
& + Y\left[\left(x^{2} + x\right) X^{10} + \left(x^{6} + x^{5} + x^{4} + x + 1\right) X^{8} +
\left(x^{8} + x^{6} + x^{5} + x^{3} + x^{2} + x + 1\right) X^{6}\right]\\
& +
Y\left[\left(x^{9} + x^{8} + x^{7} + x^{5} + x^{4} + x^{2} + x + 1\right) X^{4} +
\left(x^{10} + x^{7} + x^{5}\right) X^{2}  +x^{11}+x^{9} + x^{4}                 
\right]. 
\end{align*}

 This is proved by straightforward verification of the initial condition and the recursion identity 
 $ F_{\leq 34}^{(1)}- (g^{(1)})^7F_{\leq 34}= F_{34}^{(1)}$.
 
 Finally, similar methods as in the proof of the first case, gives $F_{<7}$ satisfying 
 $F_{<7}(d)=\ell_d^7S_{<d}(7)$ as follows

 $F_{<7} = N_{<7} /D_{<7}$, where $D _{<7} = (x^8+x)  (X^{2} + x + 1)^{7}$ and $N_{<7}$ is
 \begin{align*}
 \begin{split}
& \left(x^{8} + x^{6} + x^{5} + x^{3} + 1\right) X^{14}  + 
\left(x^{8} + x^{6} + x^{5} + x^{3} + 1\right) X^{13} \\ 
& + 
\left(x^{8} + x^{6} + x^{5} + x^{4} + x^{3} + x^{2} + 1\right) X^{12}  + 
\left(x^{9} + x^{7} + x^{6} + x^{2} + x\right) X^{11} \\
  & + 
\left(x^{12} + x^{10} + x^{4} + x^{3} y + x^{2} + y(x^8+x^6+x^5+x^3+1)\right) X^{10} +
\left(x^{12} + x^{8} + x^{7} + x^{4} + x^{3}\right) X^{9} \\ 
& + 
\left(x^{12} + x^{8}  + x^{7} + x^{5}  + x^{2} + y(x^9+x^8+x^7+x^5+x^3+x^2+x+1)\right) X^{8} \\ 
& + 
\left(x^{13} + x^{10} + x^{6} + x^{5} + x^{4} + x^{3} + x\right) X^{7} \\ 
 & + 
\left(x^{13} + x^{12} +  x^{9} + x^{7} + x^{5}  + x^{3}  + y(x^{12}+x^9+x^5+x^4+x^2+x +1)\right) X^{6} \\ 
& + 
\left(x^{14} + x^{13} + x^{12} + x^{11} + x^{9} + x^{6} + x^{4} + x^{3} + x^{2}\right) X^{5} \\ 
&  + 
\left(x^{14} + x^{13} + x^{12} + x^{11} + x^{9} + x^{8} + x^{3} y + x^{3}   + y(x^{13}+x^{12}+x^{10}+x^9     x^6+x^2+ 1)\right) X^{4} \\ 
& + 
\left(x^{15} + x^{12} + x^{11} + x^{8} + x^{6} + x^{4} + x^{3} + x^{2} + x + 1\right) X^{3} \\
& + 
\left(x^{16} + x^{15} + x^{14} + x^{13}  + x^{10} + x^{9} + x^{8} + x^{7} + x^{6}
 + x^{5} + x^{4} + y(x^{14}+x^{11}+x^{10}+x^4 +1)\right) X^{2} \\ 
& + 
\left(x^{15} + x^{14} + x^{12} + x^{10} + x^{8} + x^{6} + x^{3} + x^{2} + x\right) X \\ 
& + 
  x^{15} +  + x^{14} + x^{12}  + x^{10} + x^{9} + x^{8} + x^{7} + x^{5} + x^{4} + x^{2}\\
  & +  y( x^{15}+x^{14}+x^{12}+  x^{10}+x^8+       x^6 +x^3+x^2+x)\\
& + 
Y\left[\left(x^{8} + x^{6} + x^{5} + x^{3} + 1\right) X^{10} +\left(x^{9} + x^{8} + x^{7} + x^{5} + x^{3} +x^2+x+1\right)X^8\right]\\
& +
Y\left[ 
\left(x^{12} + x^{9} + x^{5} + x^{4} + x^{2} + x + 1\right) X^{6} + 
\left(x^{13} + x^{12} + x^{10} + x^{9} + x^{6} + x^{3} + x^{2} + 1\right) X^{4} \right]\\
& +
Y\left[  
\left(x^{14} + x^{11} + x^{10} + x^{4} + 1\right) X^{2} + 
x^{15} + x^{14} + x^{12} + x^{10} + x^{8} + x^{6} + x^{3} + x^{2} + x 
\right].
\end{split}
 \end{align*}

 The proof is thus complete, as before,  by observing the ratio of the leading terms of $F_{<7}$ and 
 $F_{\leq 34}$ is exactly (after simple cancellations) $(x^8+x^6+x^5+x^3+1)/(x^4+x^2)$. 
  \end{proof}
  
  \begin{remarks}
 (1) We solved by using SageMath,  the Frobenius-difference equation 
  $$(X^2+x+1)^7[Z^{(1)}-(X^4+x+1)N_{34}^{(1)}]=(Y+y+X^4+X^3+X^2x+X+1)^7Z,$$
  where $Z=\sum_{k=0}^{14}a_kX^k+Y \sum_{m=0}^{12} b_mX^m$, by using the elliptic curve relation 
  to get rid of higher powers of $Y$ and then equating coefficients 
  of $X^n$ and $YX^m$, 
  for $0\leq n\leq 39, \  0\leq m\leq 38$ in the resulting linear system in 26 unknowns
 $a_i, b_i$. The  unique solution obtained, in fact, proves the recursion relation. (We note here that $Z=N_{\leq 34}/(x^6+x^5+x^4+x^3+x^2+x+1)$.)
  
  (2) We would have complete case-by-case algorithmic proof method for the whole family (at least 
  in depth 2 and probably in general by induction on depth), if only 
  we are assured of solvability of such equations resulting from our recursion.
  \end{remarks}

  \section{Dedekind type relative zeta situation}
  
  We also consider  Dedekind type relative zeta and multizeta functions using norms from 
  $A$ to some corresponding $\F_q[x]$, say and explore corresponding zetalike multizetas. 
  
  More precisely, for a monic $a\in A$, we use the monic generator of $-k$-th power of 
  the relative norm of $a$. For the class number one situation, this corresponds more closely to 
  the Dedekind zeta. See \cite{T04}[Sec. 5.1].
  
  First note that if the relative extension is Galois of degree $p$ and $A$ is class number 
  one, then (argument of  \cite{T04}[Pa 162] generalized to power sums) for an element of $A-\F_q[x]$, the $p$ conjugates having 
  the same norm, the total norm contribution is zero, where as for an element in the base, 
  the norm is $p$-th power, so  $\zeta_{A/\F_q[x]}(s_1, \cdots, s_r)
  =\zeta_{\F_q[x]}(ps_1, \cdots, ps_r)$. In particular,  we get  zetalike elements just from the genus zero case. 
  This works for 
 the three class number one examples $A$ with $p=2$, which are quadratic 
  over $\F_q[x]$ of the 
  form $y^2+y=P(x)$, and the fourth class number one example with $q=3$ of 
  the form $y^2=x^3-x-1$ considered as a cubic Galois extension (since 
  $\F_3$ translations of a root are roots) over $\F_3[y]$.

  Considered the  class number one examples above of characteristic 2,  as extensions of the relevant $\F_q[y]$'s, we did not find any zetalike examples, 
  in numerical experimentation.   
Similarly, for $q = 3, y^2 = x^3-x-1$, and $Norm (f+yg) = f^2 - y^2g^2$, we have not yet 
found any zetalike examples.

  We are in the beginning stages of exploration in the general relative situation and will report 
in the future paper about more refined conjectures on degrees,   other F-functions and relations.

For now, we only make following simple remark that  in higher genus, some power sums are zero, 
not only as the relevant sets are empty because of Weierstrass gaps (at the point at infinity), but 
also  power sums can be zero, even if the relevant sets
are not empty. For example,  consider $\F_2[x, y]/y^2+y=x^3+x+1$ over 
$\F_2[y]$. In this case, since the norm of $x$ as well as of $x+1$ is $y^2+y+1$, all 
the power sums for degree $2$ also (for degree $1$ they vanish for the reason above) vanish.

  \section{Numerical experiments}
  
  The numerical exploration to find zetalike values was done following the method of \cite{LT14} using SageMath on 
  laptop,  using 
  the continued fractions  in $\F_q((1/x))$.
Note that in cases (i, iii, iv), $S_d(k)\in \F_q(x)$ by invariance with respect to 
the Galois action $y\rightarrow y+1$ of $K$ over $\F_q(x)$. In the case 
(ii), if $s_i$'s are even, we  get the relevant Galois invariance. In these cases, the method 
of \cite{LT14} using continued fractions for $\F_q((1/x))$ works immediately. 
In case (ii), in general, and in higher class number cases, (for low $q, g$),  we used the norms to descend to this 
$\F_q((1/x))$ situation. 

Apart from higher class number and Dedekind situation, we also looked  for possible rational 
 ratios of multizeta (of depth 2 or 3) of the same weight in class number one case, not explained 
by our conjecture on the zetalike family. 
 We did not find any, in contrast to several examples in genus zero, such as, when $A=\F_2[t]$, 
 we have rational ratios
 $\zeta(1, 3)/\zeta(2, 2)$, $\zeta(2, 3)/\zeta(3 2)$, $\zeta(7, 4)/\zeta(4, 7)$, where only for the first 
 numerators and denominators are zetalike (so the first is explained by zetalikes), but none of the numerators or denominators of the  last 2 are zetalike. 
 (We checked the class number one cases (i), (iv)  for weights up to 32, and (iii) for weights up to 
 12, only for depth 2).

{\bf Acknowledgments.} 
The second author thanks Max Planck Institute in Mathematics, Bonn for its support 
when he conducted this research. Theorem 3.4  was  first announced in October 2019 
at his MPI seminar.

\newpage

\appendix

\section{Case ii, $n = 1$, for $A = \F_3[x,y]$ with $y^2  = x^3-x-1$}

\begin{theorem}
For $A = \F_3[x,y]/ (y^2 = x^3-x-1)$, we have
\begin{align*}
(x^{9} + x^{6} + x^{4} - x^{3} + x^{2}  -1)\zeta(2,6)  = (x^{3} - x + 1)\zeta(8).
\end{align*}
\end{theorem}

\begin{proof}

We proceed in a similar way to the proof of case i). 
We will now define several functions in $\F_3(x,y,X,Y)$, where $x$ and $X$ are independent transcendentals and $y^2 = x^3-x-1$ and $Y^2 = X^3-X-1$. For each function, say $h$, put $h^{(1)}$ for the function resulting from $h$ after substituting $X^3, Y^3$ respectively for $X,Y$, and put $h(d) \in K$ for the function resulting from $h$  after substituting $x^{3^d}$ and $y^{3^d}$ for $X,Y$. 
Put 
\begin{gather*}
F_1  =  \frac{-(X-x^q) }{X^{2} + \left(x + 1\right) X + y Y + x^{2} - x + 1}, \qquad
 g^3  = \frac{ -(Y-y)^3 + Y^q (X-x)^3} { Y (Y^3-Y) -(X-x^3) - Y^3  Y  (X^3 - X)},\\
 F_{<2}  = -yg^3 + F_1^3 - F_1^2, \qquad 
F_{26}  = F_1^2 F_{<2}^3,\\
F_{<8}  = - \frac{(x^3-x+1)g^9 (g^{(-1)})^9}{x^3-x} + (yg^3 - F_1^3 + F_1^2)^4 + (F_1^3-F_1^2)(y^3 g^9 -F_1^9+F_1^6),\\
A_1  = F_1 (yg^3 - F_1^3 + F_1^2), \quad 
A_2  = F_1 \left( \frac{x^{3} - x + 1}{x^{3}- x} \cdot {g^{9} (g^{(-1)})^{9} } -(F_1^3- F_1^2)  (y^3 g^{9} -F_1^9+F_1^6)  \right).  
\end{gather*}

Notice that $A_1(d) = \ell_d^3 A_{d1}$ and $A_2(d) = \ell_d^{9}A_{d2}$. 
Then, we have $\ell_d S_d(1) = \ell_d A_{d0} = F_1(d)$ and 
\begin{align*}
 \ell_d^2 S_{<d}(2)& = \frac{-\ell_d^3 A_{d1} }{\ell_d A_{d0}} = \frac{ A_1(d)}{F_1(d)} 
  = (-yg^3 + F_1^3 - F_1^2)(d) = F_{<2}(d), \\
 \ell_d^8 S_d(2,6) & = (\ell_d^2 S_d(2) )( \ell_d^2 S_{<d}(2))^3
  = F_1(d)^2 F_{<2}(d)^3
  = F_{26}(d).
\end{align*}
Finally, we have 
\begin{align*}
%
\ell_d^8 S_{<d}(8) & = 
\frac{ (\ell_d^3 A_{d1})^4 - ( \ell_d A_{d0})^3 \cdot \ell_d^9 A_{d2}   }{ (\ell_d A_{d0})^4} = \frac{ A_1(d)^4 - F_1(d)^3 A_2(d)   }{F_1(d)^4}\\
& = - \frac{(x^3-x+1)g^9 (g^{(-1)})^9}{x^3-x} + (yg^3 - F_1^3 + F_1^2)^4 + (F_1^3-F_1^2)(y^3 g^9 -F_1^9+F_1^6), 
\\
& = F_{<8}(d).
\end{align*}

Explicitly $F_{26} = N_{26}/D_{26}$, where $D_{26} = (X^3-x+1)^8$ and $N_{26}$ is 
\begin{align*}
&\left(x^{6} + x^{4} + x^{3} + x^{2} - x + 1\right) X^{17} + 
\left(x^{3} y - x y - y\right) X^{16} Y + 
\left(x^{7} + x^{6} + x^{5} - x^{4} - x^{3}\right) X^{16} +\\&
\left(x^{6} y - x^{2} y + x y - y\right) X^{15} Y + 
\left(x^{8} - x^{7} - x^{4} - x^{3} + x^{2} - x + 1\right) X^{15} + 
\left(-x^{3} y + x y\right) X^{14} Y + \\&
\left(x^{7} + x^{6} + x^{5} - x^{4} - x^{3} + 1\right) X^{14} + 
\left(-x y - y\right) X^{13} Y + 
\left(-x^{9} + x^{8} - x^{7} - x^{4} + x^{3} + x^{2} - x\right) X^{13} + \\& 
\left(x^{7} y + x^{6} y - x^{5} y + x^{4} y - x^{3} y - x^{2} y + x y - y\right) X^{12} Y + \\&
\left(-x^{12} - x^{10} - x^{9} + x^{7} + x^{6} + x^{5} + x^{4} + x^{2} + x + 1\right) X^{12} + \\&
\left(-x^{9} y - x^{6} y + x^{3} y + x^{2} y\right) X^{11} Y + 
\left(-x^{9} + x^{8} - x^{7} - x^{4} + x^{3} + x^{2}\right) X^{11} + \\&
\left(-x^{10} y - x^{9} y - x^{7} y + x^{6} y + x^{5} y - x^{4} y - x^{3} y\right) X^{10} Y + \\&
\left(-x^{12} + x^{10} - x^{9} + x^{7} - x^{6} + x^{5} + x^{4} - x^{2} + 1\right) X^{10} + \\&
\left(-x^{11} y + x^{10} y + x^{9} y + x^{7} y - x^{6} y - x^{5} y + x^{4} y - x^{3} y\right) X^{9} Y + \\& 
\left(-x^{13} - x^{11} - x^{10} + x^{9} + x^{8} + x^{7} + x^{5} + x^{3} - x^{2} + x + 1\right) X^{9} + 
\left(-x^{6} y - x^{4} y - x^{2} y + y\right) X^{8} Y + \\&
\left(-x^{12} + x^{10} + x^{9} - x^{7} + x^{6} - x^{5} + x^{4} - x^{3} + x^{2} + x + 2\right) X^{8} + \\&
\left(-x^{9} y - x^{7} y - x^{6} y - x^{5} y + x^{4} y - x y - y\right) X^{7} Y +\\& \left(-x^{13} - x^{12} + x^{11} - x^{10} - x^{8} - x^{6} - x^{4} + x^{3} + x^{2} + x + 1\right) X^{7} + \\&
\left(-x^{12} y - x^{10} y - x^{9} y - x^{8} y - x^{6} y + x^{5} y - x^{4} y + x^{3} y - x^{2} y + x y - y\right) X^{6} Y + \\&
\left(-x^{14} + x^{13} - x^{12} - x^{10} - x^{8} + x^{3} + x + 2\right) X^{6} + \left(x^{9} y - x^{7} y - x^{6} y - x^{5} y + x^{2} y - x y + y\right) X^{5} Y +\\&
\left(-x^{13} + x^{12} + x^{11} - x^{8} + x^{6} - x^{4} - x^{3} - x^{2} - x\right) X^{5} +\\& \left(-x^{9} y - x^{8} y + x^{7} y - x^{6} y + x^{5} y + x^{4} y - x^{3} y + x y + y\right) X^{4} Y + \\&
\left(-x^{14} + x^{11} - x^{10} - x^{9} - x^{8} - x^{7} + x^{6} - x^{4} - x^{3} + x^{2} - x + 1\right) X^{4} + \\&
\left(-x^{13} y + x^{12} y - x^{10} y - x^{9} y - x^{8} y + x^{7} y + x^{6} y - x^{5} y + x^{4} y - x^{3} y + x^{2} y - x y + y\right) X^{3} Y + \\&
\left(-x^{15} - x^{14} + x^{13} - x^{12} - x^{11} - x^{10} + x^{8} + x^{6} - x^{4} + x^{3} - x^{2} + x\right) X^{3} + \\&
\left(-x^{8} y - x^{7} y + x^{5} y + x^{4} y + x^{3} y + x^{2} y + x y + y\right) X^{2} Y + \\&
\left(-x^{14} - x^{13} + x^{12} - x^{11} + x^{10} + x^{9} - x^{8} + x^{7} - x^{6} - x^{3} + x^{2} + 1\right) X^{2} + \\&
\left(-x^{11} y - x^{10} y - x^{9} y + x^{8} y - x^{7} y - x^{6} y + x^{5} y + x^{4} y + x^{3} y\right) X Y + \\&
\left(-x^{15} + x^{14} + x^{11} - x^{9} + x^{7} - x^{6} + x^{5} + x^{3}\right) X +\\&
\left(-x^{14} y - x^{13} y - x^{12} y + x^{11} y + x^{10} y + x^{9} y - x^{8} y - x^{6} y\right) Y - x^{16} - x^{14} + x^{13} - x^{12}\\& + x^{11} - x^{10} + x^{9} + x^{8} - x^{7}. 
\end{align*}

In order to find $F_{ \le 26}$, such that $F_{\le 26}(d) = \ell_d^8 S_{\le d}(2,6)$, we use the recursion identity $F_{\le 26}^{(1)} - (g^{(1)})^8 F_{\le 26} = F_{26}$. Let $Z = (X^3-x+1)^8 F_{\le 26}$. We solve by using SageMath, the equation 
\begin{align*}
 (X^3-x+1)^8 (Z^{(1)} - N_{26}^{(1)}  ) = (-X^6Y + X^4 Y + (x - 1) X^3 Y + (-x - 1) X Y + (-x - 1) Y - y )^8 Z,
\end{align*}
where $Z = \sum _{k=0}^{18} a_k X^k + Y \sum _{m = 0}^{15} b_m X^m$. The unique solution obtained is $Z = N_{\le 26}/(x^3-x)$, so that 
$F_{\le 26} = N_{ \le 26}/D_{\le 26}$ where $D_{\le 26 } = (x^3-x)(X^3-x+1)^8$ and
$N_{\le 26}$ is
\begin{align*}
&\left(x^{3} - x + 1\right) X^{18} + \left(-x^{6} y - x^{4} y + x^{3} y - x^{2} y - x y\right) X^{15} Y + 
\left(-x^{9} - x^{6} - x^{4} + x^{3} - x^{2} + 1\right) X^{15} + \\&
\left(x^{6} y + x^{4} y - x^{3} y + x^{2} y + x y\right) X^{13} Y + \left(-x^{9} y + x^{7} y - x^{6} y + x^{5} y + x^{4} y + x^{3} y + x y\right) X^{12} Y + \\&
\left(-x^{10} + x^{9} - x^{7} + x^{6} - x^{5} - x^{4} + x^{2} - x\right) X^{12} + \left(x^{9} y - x^{7} y - x^{6} y - x^{5} y + x^{3} y + x^{2} y\right) X^{10} Y + \\&
\left(-x^{7} y - x^{6} y - x^{5} y - x^{3} y + x^{2} y + y\right) X^{9} Y + \left(-x^{11} - x^{10} + x^{9} - x^{8} - x^{7} + x^{6} + x^{4} - x^{3} + x^{2} + 1\right) X^{9} + \\&
\left(x^{9} y - x^{3} y - y\right) X^{7} Y + \left(-x^{10} y + x^{7} y + x^{6} y + x^{5} y + x^{4} y - x^{3} y - x^{2} y + y\right) X^{6} Y + \\&
\left(-x^{12} - x^{10} + x^{7} + x^{5} - x^{3} - x^{2} + x + 1\right) X^{6} + \left(x^{10} y + x^{9} y - x^{7} y - x^{6} y - x^{5} y - x^{4} y + x^{2} y + y\right) X^{4} Y + \\&
\left(-x^{11} y - x^{9} y - x^{8} y - x^{6} y - x^{5} y - x^{4} y + x^{2} y + y\right) X^{3} Y + \\&
\left(-x^{13} - x^{12} - x^{11} - x^{9} + x^{8} - x^{7} + x^{6} - x^{5} - x^{4} + x^{3} - x^{2} - x + 1\right) X^{3} +\\& \left(x^{11} y + x^{10} y - x^{9} y + x^{8} y - x^{7} y\right) X Y + \\&
\left(x^{11} y + x^{10} y - x^{9} y + x^{8} y - x^{7} y\right) Y + x^{15} - x^{14} - x^{13} + x^{12} + x^{11} + x^{10} + x^{9} + x^{8} + x^{7}.
\end{align*}

$F_{<8} = N_{<8}/D_{<8}$ where $D_{< 8} = (x^3-x)(X^3-x+1)^8$ and $N_{<8}$ is 
\begin{align*}
&\left(x^{9} + x^{6} + x^{4} - x^{3} + x^{2} + 2\right) X^{18} + 
\left(-x^{3} y + x y\right) X^{15} Y - X^{15} + 
\left(x^{3} y - x y\right) X^{13} Y + \\&
\left(-x^{12} y + x^{10} y - x^{4} y + x^{3} y + x^{2} y - x y\right) X^{12} Y +\\& \left(x^{9} + x^{6} + x^{4} - x^{3} + x^{2} - x\right) X^{12} + \left(x^{12} y - x^{10} y + x^{4} y - x^{2} y\right) X^{10} Y + \\&
\left(x^{12} y - x^{10} y + x^{4} y + x^{3} y - x^{2} y - x y - y\right) X^{9} Y + \left(-x^{9} - x^{6} - x^{4} + x^{3} + x^{2} - x + 2\right) X^{9} + \\&
\left(-x^{3} y + x y + y\right) X^{7} Y + 
\left(x^{12} y - x^{10} y + x^{4} y - x^{3} y - x^{2} y - y\right) X^{6} Y + \left(x^{9} + x^{6} + x^{4} + x^{3} + x^{2} - x + 2\right) X^{6} + \\&
\left(-x^{12} y + x^{10} y - x^{4} y + x^{2} y + x y - y\right) X^{4} Y + 
\left(-x^{12} y + x^{10} y - x^{4} y - y\right) X^{3} Y + \\&
\left(-x^{4} + x^{3} - x^{2} + x + 2\right) X^{3} + \left(x^{2} y + x y\right) X Y + \left(x^{2} y + x y\right) Y + x^{18} + x^{15} + x^{13} - x^{12} + x^{11} + x^{6} - x^{5} + x.
\end{align*}

Notice that denominators of $F_{\le 26}$ and $F_{ <8}$ match. The degree of $E = (x^{9} + x^{6} + x^{4} - x^{3} + x^{2} -1)F_{\le 26} - (x^{3} - x + 1)F_{<8}$ is negative; more precisely, the degree of $E(d)$ is  $-(-27+15 \times 3^d)$; since the degree of $\ell_d$ is $-(3^{d+1}-3)/2$, the degree of the $E(d)/\ell_d^8$ is $-(-15+3^{d+1})$; therefore, $E(d)/\ell_d^8$ tends to zero as $d$ tends to infinity. 

\end{proof}

\section{Case iii, $n=1$, $A = \F_4[x,y]$ with $y^2 +y = x^3 +w$}

\begin{theorem}
\begin{align*}
\zeta(3, 12) / \zeta(15)  = (x^{12} + x^9 + x^6 + x^3 + 1)/(x^{24} + x^{18} + x^9 + x^3 + 1) 
\end{align*}
\end{theorem}

\begin{proof}
We will now define several functions in $\F_3(x,y,X,Y)$, where $x$ and $X$ are independent transcendentals and $y^2 +y= x^3+w$ and $Y^2 = X^3+w$. For each function, say $h$, put $h^{(1)}$ for the function resulting from $h$ after substituting $X^4, Y^4$ respectively for $X,Y$, and put $h(d) \in K$ for the function resulting from $h$  after substituting $x^{4^d}$ and $y^{4^d}$ for $X,Y$. 
Put 
\begin{gather*}
 F_1  =\frac{X+x^q}{x^{2} X^{2} + X Y + \left(y + 1\right) X + x Y + x y} ,\qquad
 g  = \frac{Y+y  +X^2(X-x)}{X^4+x} \\
F_{<3}  = (x^4+x)g^4 -F_1^4 + F_1^3, \qquad 
F_{3, 12}  = F_1^3 F_{<3}^4,  \qquad
A_1 = F_1 ((x^4+x)g^4 +F_1^4 + F_1^3 ),\\
A_2 = F_1 ( \frac{(x^{12} + x^9 + x^6 + x^3 + 1) g^{16}(g^{(-1)})^{16} }{x^4 + x}  + (F_1^4 + F_1^3)((x^4+x)g^4+ F_1^4 + F_1^3 )^4 ),
\end{gather*}
and
\begin{multline*}
F_{<15}  = ((x^4+x)g^4 + F_1^4 + F_1^3)^{5} +\\ \frac{(x^{12} + x^9 + x^6 + x^3 + 1){g^{16} (g^{(-1)})^{16}  }}{x^4+x}  (F_1^4+F_1^3)( (x^4+x)g^4 + F_1^4 + F_1^3)^4.
\end{multline*}
Note that $A_1(d) = \ell_d ^4 A_{d1}$ and $A_2(d) = \ell_d^{16} A_{d2}$. 

We have 
\begin{align*}
\ell_d S_d(1) & = F_1(d)\\
 \ell_d^3 S_{<d}(3) &= \frac{\ell_d^4 A_{d1}  }{\ell_d A_{d0}}  
= \frac{F_1(d) ( (x^4+x)g^4 - F_1^4 + F_1^3  )(d)}{F_1(d)}
= F_{<3}(d),\\
\ell_d^{15} S_d(3,12) & = ( \ell_dS_d(3))^3 (\ell_d^3 S_{<d}(3))^4 = F_1(d)^3 F_{<3}(d)^4 = F_{3,12}(d),\\
 \ell_d^{15} S_{<d}{15} &=  \frac{(\ell_d^4 A_{d1} )^5 + ( \ell_d A_{d0} )^4 ( \ell_d^{16} A_{d2} )   }{ (\ell_d A_{d0})^5  } = F_{<15}(d).   
\end{align*}

Explicitly, we have  $F_{<15} = N_{<15}/ D_{<15}$, where $D_{<15} = (x^4+x) (X^4+x)^{15}$ and 

\begin{align*}
N_{<15} & = \left(x^{24} + x^{18} + x^{9} + x^{3} + 1\right) X^{40} + \left(x^{8} + x^{2}\right) X^{38} + \left(x^{16} + x^{10} + x^{4} + x\right) X^{36} + \left(x^{8} + x^{2}\right) X^{35}\\&
- X^{34} + \left(x^{8} + x^{2}\right) X^{32} Y + \left(x^{32} + x^{26} + x^{20} + x^{17} + x^{8} y + x^{2} y\right) X^{32} - X^{31} + x X^{30} - X^{28} Y + \\&
\left(y + 1\right) X^{28} + x X^{27} + x^{2} X^{26} + x X^{24} Y + \left(x y + x\right) X^{24} + x^{2} X^{23} + x^{3} X^{22} + x^{2} X^{20} Y + \\&
\left(x^{2} y + x^{2}\right) X^{20} + x^{3} X^{19} + x^{4} X^{18} + x^{3} X^{16} Y + \left(x^{24} + x^{18} + x^{9} + x^{3} y + 1\right) X^{16} + x^{4} X^{15} + \\&
\left(x^{8} + x^{5} + x^{2}\right) X^{14} + x^{4} X^{12} Y + \left(x^{16} + x^{10} + x^{4} y + x\right) X^{12} + \left(x^{8} + x^{5} + x^{2}\right) X^{11} + \left(x^{6} + 1\right) X^{10} + \\&
\left(x^{8} + x^{5} + x^{2}\right) X^{8} Y + \left(x^{32} + x^{26} + x^{20} + x^{17} + x^{8} y + x^{5} y + x^{5} + x^{2} y\right) X^{8} + \left(x^{6} + 1\right) X^{7} + \\&
\left(x^{7} + x\right) X^{6} + \left(x^{6} + 1\right) X^{4} Y + \left(x^{24} + x^{18} + x^{9} + x^{6} y + x^{6} + x^{3} + y\right) X^{4} + \left(x^{7} + x\right) X^{3} + \\&
\left(x^{7} + x\right) Y + x^{40} + x^{34} + x^{25} + x^{19} + x^{10} + x^{7} y + x^{7} + x^{4} + x y.
\end{align*}

On the other hand, 
$F_{3,12} = N_{3,12}/D_{3,12}$ where $D_{3,12} = (X^4+x)^{15}$ and $N_{3,12}$ is  (too large to fit in here)
\begin{align*}
N_{3,12} & = (x^{18} + x^6)X^{39} + (x^{16} + x^4)X^{38}Y + (x^{22} + x^{16}y + x^{16} + x^{10} + x^4y + x^4)X^{38} + \dotsb
\end{align*}

Next, We calculate $F_{\le 3,12}$ such that $F_{\le 3,12}(d) = \ell_d^{15}S_{\le d}(3,15)$. By using the  identity recursion $F_{\le 3,12} - (g^{(1)})^{15}F_{\le 3,12} = F_{3,12}$,   
we get the Frobenius difference equation 
\begin{align*}
(X^4+x)^{15} ( Z^{(1)}-  N_{3,12}^{(1)}) & = (X^{12} + xX^8 + X^6 + X^3 + Y + y + 1)^{15}Z,
\end{align*}
where $Z = (X^4+x)^{15} F_{\le 3,12}$. We put $Z = \sum _{k=0}^{40} a_k X^k + Y \sum _{m = 0}^{36} b_m X^m$. The unique solution obtained is $Z = N_{\le 3,12}/(x^4+x)$ so that 
$F_{\le 3,12} = Z/(X^4+x)^{15} = N_{\le 3,12}/D_{\le 3,12}$ where $D_{\le 3,12} = (x^4+x)(X^4+x)^{15}$ and $N_{\le 3,12}$ is 
\begin{align*}
&\left(x^{12} + x^{9} + x^{6} + x^{3} + 1\right) X^{40} + \left(x^{20} + x^{17} + x^{8} + x^{5}\right) X^{38} + \left(x^{22} + x^{19} + x^{10} + x^{7}\right) X^{36} + \\&
\left(x^{20} + x^{17} + x^{8} + x^{5}\right) X^{35} + \left(x^{24} + x^{18} + x^{9} + x^{3} + 1\right) X^{34} + \left(x^{20} + x^{17} + x^{8} + x^{5}\right) X^{32} Y + \\&
\left(x^{26} + x^{23} + x^{20} y + x^{17} y + x^{14} + x^{11} + x^{8} y + x^{5} y\right) X^{32} + \left(x^{24} + x^{18} + x^{9} + x^{3} + 1\right) X^{31} + \\&
\left(x^{25} + x^{19} + x^{10}\right) X^{30} + \left(x^{24} + x^{18} + x^{9} + x^{3} + 1\right) X^{28} Y + \\&
\left(x^{24} y + x^{24} + x^{18} y + x^{18} + x^{9} y + x^{9} + x^{6} + x^{3} y + y + 1\right) X^{28} + \left(x^{25} + x^{19} + x^{10}\right) X^{27} + \\&
\left(x^{26} + x^{20} + x^{11} + x^{8} + x^{5}\right) X^{26} + \left(x^{25} + x^{19} + x^{10}\right) X^{24} Y + \left(x^{25} y + x^{25} + x^{19} y + x^{19} + x^{10} y + x^{7}\right) X^{24} + \\&
\left(x^{26} + x^{20} + x^{11} + x^{8} + x^{5}\right) X^{23} + \left(x^{27} + x^{21} + x^{6}\right) X^{22} + \left(x^{26} + x^{20} + x^{11} + x^{8} + x^{5}\right) X^{20} Y + \\&
\left(x^{26} y + x^{26} + x^{20} y + x^{20} + x^{14} + x^{11} y + x^{8} y + x^{8} + x^{5} y + x^{5}\right) X^{20} + \left(x^{27} + x^{21} + x^{6}\right) X^{19} + \\&
\left(x^{28} + x^{22} + x^{16} + x^{13} + x^{7}\right) X^{18} + \left(x^{27} + x^{21} + x^{6}\right) X^{16} Y + \\&
\left(x^{27} y + x^{27} + x^{21} y + x^{21} + x^{18} + x^{15} + x^{12} + x^{9} + x^{6} y + x^{3} + 1\right) X^{16} + \left(x^{28} + x^{22} + x^{16} + x^{13} + x^{7}\right) X^{15} + \\&
\left(x^{29} + x^{23} + x^{17} + x^{14} + x^{5}\right) X^{14} + \left(x^{28} + x^{22} + x^{16} + x^{13} + x^{7}\right) X^{12} Y + \\&
\left(x^{28} y + x^{28} + x^{22} y + x^{22} + x^{16} y + x^{16} + x^{13} y + x^{13} + x^{10} + x^{7} y\right) X^{12} + \left(x^{29} + x^{23} + x^{17} + x^{14} + x^{5}\right) X^{11} + \\&
\left(x^{30} + x^{24} + x^{18} + x^{15} + x^{3} + 1\right) X^{10} + \left(x^{29} + x^{23} + x^{17} + x^{14} + x^{5}\right) X^{8} Y + \\&
\left(x^{29} y + x^{29} + x^{23} y + x^{23} + x^{20} + x^{17} y + x^{14} y + x^{11} + x^{8} + x^{5} y\right) X^{8} + \left(x^{30} + x^{24} + x^{18} + x^{15} + x^{3} + 1\right) X^{7} + \\&
\left(x^{31} + x^{28} + x^{19} + x^{16}\right) X^{6} + \left(x^{30} + x^{24} + x^{18} + x^{15} + x^{3} + 1\right) X^{4} Y + \\&
\left(x^{30} y + x^{27} + x^{24} y + x^{24} + x^{18} y + x^{18} + x^{15} y + x^{15} + x^{12} + x^{9} + x^{3} y + x^{3} + y\right) X^{4} + \\&
\left(x^{31} + x^{28} + x^{19} + x^{16}\right) X^{3} + \left(x^{31} + x^{28} + x^{19} + x^{16}\right) Y + x^{34} + x^{31} y + x^{28} y + x^{25} + x^{19} y + x^{19} + x^{16} y.
\end{align*}

The leading term of the numerator of  $E = (x^{24} + x^{18} + x^9 + x^3 + 1) F_{\le 3, 12} -(x^{12} + x^9 + x^6 + x^3 + 1) F_{<15}$ is 
$$
(x^{44} + x^{41} + x^{38} + x^{35} + x^{32} + x^{20} + x^{17} + x^{14} + x^{11} + x^2)X^{38}
$$
so that the degree of $E(d)$ is $-(-80+44q^d)$; since the degree of $\ell_d$ is $-8(4^d-1)/3$,  the degree of $E(d)/\ell_d^{15}$ is $-(-40+4^{d+1})$  and thus  as $d$ tends to infinity, the error tends to zero.
\end{proof}


\section{Case iv, $n = 1$, $A = \F_2[x,y]$ with $y^2 + y = x^5+x^3+1$}

\begin{theorem}
For $A = \F_2[x,y]$ with $y^2 + y = x^5+x^3+1$, we have
\begin{align*}
(x^8 + x^6 + x^5 + x^4 + x^3 + x + 1)\zeta(1,2) = (x^6 + x^5 + x^3 + x + 1) \zeta(3). 
\end{align*}
\end{theorem}
\begin{proof}

We will now define several functions in $\F_2(x,y,X,Y)$, where $x$ and $X$ are independent transcendentals and $y^2 +y= x^5+x^3+1$ and $Y^2 +Y = X^5+X^3+1$. For each function, say $h$, put $h^{(1)}$ for the function resulting from $h$ after substituting $X^2, Y^2$ respectively for $X,Y$, and put $h(d) \in K$ for the function resulting from $h$  after substituting $x^{2^{d-2}}$ and $y^{2^{d-2}}$ for $X,Y$. 
Put $F_1 = N_1 / D_1$, where $D_1 = (X^8+x)(X^{16}+x+1)$ and 
\begin{align*}
N_1 & = X^{15} + x^{2} X^{14} + x X^{13} + \left(x^{3} + x\right) X^{12} + x X^{11} + X^{10} Y + \left(x^{2} + x + y\right) X^{10} + \left(x + 1\right) X^{9}\\& + \left(x + 1\right) X^{8} Y + \left(x^{2} + x y + y + 1\right) X^{8} + x X^{7} + X^{6} Y + \left(x^{2} + y\right) X^{6} + x X^{5} + x X^{4} Y \\& + \left(x^{3} + x y + x\right) X^{4} + x X^{3} + \left(x^{3} + x^{2}\right) X^{2} + x Y + x^{4} + x y + x.
\end{align*}

By Conjecture E in \cite[pa. 194]{T92} (which is now a theorem, see \cite{A94}), we get that 
$\ell_d S_d(1) = F_1(d)$. 
Next we will calculate $F_{<1}$ in such a way that $F_{<1}(d) = \ell_d S_{<d}(1)$. 
From the identity $S_{<d+1}(1) = S_{<d}(1) + S_d(1)$ we get
\begin{align*}
\ell_{d+1} S_{<d+1}(1) &= \frac{\ell_{d+1}}{\ell_d} (\ell_d S_{<d}(1) + \ell_d S_d(1))\\
F_{<1}(d+1)& = g(d+1)(F_{<1}(d) + F_1(d) ).
\end{align*}
The Frobenius equation to be solved is $F_{<1}^{(1)} = g^{(1)} (F_{<1} + F_1)$. Let $Z = (X^8+x)(X^{16}+x+1) F_{<1}$. The equation to be solved
now is 
\begin{align*}
 (X^8+x) (X^{16}+x+1)Z^{(1)}  = (N_g^{(1)})(Z+N_1).
\end{align*}
We put $Z = \sum_{i=0}^{16} a_iX^i + Y \sum _{j=0}^{10} b_jX^j$, and it is obtained a system of $107$ equations with $28$ unknowns.  
The unique solution is $Z = N_{<1}$, so that $F_{<1} = N_{<1}/((X^8+x)(X^{16}+x+1) )$, where
\begin{align*}
N_{<1} &= x^2 + x)X^{16} + X^{15} + (x^2 + 1)X^{14} + xX^{13} + (x^3 + x^2 + x + 1) X^{12} + xX^{11} + X^{10} Y\\& 
+ (x^2 + y) X^{10} + x X^9 + (x + 1) X^8 Y + (x^3 + x^2 + x y + y) X^8 + (x + 1) X^7 + X^6 Y\\& + (x^2 + x + y + 1) X^6 
+ X^5 + (x + 1) X^4 Y + (x^3 + x y + x + y) X^4 + X^3 + (x^3 + x^2) X^2 + Y \\&+ x^4 + x^2 + y.
\end{align*}

Next we define $F_{12} = F_1 F_{<1}^2$. Then 
\begin{align*}
 \ell_d^3 S_d(1,2) = (\ell_d S_d(1)) (\ell_d S_{<d}(1))^2 
 = F_1(d) F_{<1}(d)^2 = F_{12}(d).
\end{align*}
Explicitly, $F_{12} = N_{12}/D_{12}$, where $D_{12} = (X^8 + x)^3  (X^{16} + x + 1)^3$ and $N_{12}$ is
\begin{align*}
&\left(x^{4} + x^{2}\right) X^{47} + \left(x^{6} + x^{4}\right) X^{46} + \left(x^{5} + x^{3} + 1\right) X^{45} + \left(x^{7} + x^{3} + x^{2}\right) X^{44} + \left(x^{5} + x^{4} + x^{3} + x + 1\right) X^{43} \\& + \left(x^{4} + x^{2}\right) X^{42} Y + \left(x^{5} + x^{4} y + x^{4} + x^{2} y + x^{2} + x\right) X^{42} + \left(x^{4} + x^{3}\right) X^{41} + \left(x^{5} + x^{4} + x^{3} + x^{2} + 1\right) X^{40} Y \\& + \left(x^{7} + x^{6} + x^{5} y + x^{5} + x^{4} y + x^{4} + x^{3} y + x^{3} + x^{2} y + y + 1\right) X^{40} + \left(x^{6} + x^{4}\right) X^{39} + \left(x^{2} + x\right) X^{38} Y  \\& + \left(x^{8} + x^{6} + x^{3} + x^{2} y + x^{2} + x y\right) X^{38} + \left(x^{7} + x^{5} + x^{4} + 1\right) X^{37} + \left(x^{4} + x^{3} + x^{2} + x\right) X^{36} Y  \\& + \left(x^{9} + x^{7} + x^{6} + x^{4} y + x^{4} + x^{3} y + x^{2} y + x^{2} + x y + x\right) X^{36} + \left(x^{7} + x^{2} + x + 1\right) X^{35} \\&  + \left(x^{6} + x^{3} + x^{2} + x\right) X^{34} Y   + \left(x^{8} + x^{7} + x^{6} y + x^{4} + x^{3} y + x^{2} y + x y + x\right) X^{34} + \left(x^{7} + x^{4} + x^{3} + x^{2}\right) X^{33}  \\& + \left(x^{7} + x^{6} + x^{5} + x^{4} + x^{3} + x^{2} + x + 1\right) X^{32} Y  \\& + \left(x^{8} + x^{7} y + x^{6} y + x^{6} + x^{5} y + x^{5} + x^{4} y + x^{3} y + x^{2} y + x^{2} + x y + x + y + 1\right) X^{32} + \left(x^{4} + x\right) X^{31}  \\& + \left(x^{6} + x^{5} + x^{4} + x^{3}\right) X^{30} Y + \left(x^{9} + x^{6} y + x^{6} + x^{5} y + x^{4} y + x^{3} y\right) X^{30} + \left(x^{8} + x^{6} + x^{4} + x^{3} + 1\right) X^{29}  \\& + \left(x^{7} + x^{6} + x^{5} + x^{2}\right) X^{28} Y   + \left(x^{10} + x^{7} y + x^{7} + x^{6} y + x^{6} + x^{5} y + x^{3} + x^{2} y + x^{2}\right) X^{28}  \\& + \left(x^{8} + x^{6} + x^{3} + x + 1\right) X^{27}   + \left(x^{7} + x^{6} + x^{5} + x^{4} + x^{2} + x\right) X^{26} Y  \\& + \left(x^{8} + x^{7} y + x^{7} + x^{6} y + x^{6} + x^{5} y + x^{4} y + x^{2} y + x y + x\right) X^{26}   + \left(x^{8} + x^{6} + x^{5} + x^{4} + x^{3} + x + 1\right) X^{25} \\&  + \left(x^{8} + x^{7} + x^{5} + x^{3} + x^{2} + x + 1\right) X^{24} Y \\&  + \left(x^{10} + x^{9} + x^{8} y + x^{8} + x^{7} y + x^{7} + x^{5} y + x^{4} + x^{3} y + x^{2} y + x y + x + y + 1\right) X^{24}  \\& + \left(x^{8} + x^{6} + x^{3} + x^{2} + 1\right) X^{23} + \left(x^{7} + x^{6} + x^{5} + x^{2}\right) X^{22} Y + \left(x^{8} + x^{7} y + x^{6} y + x^{5} y + x^{5} + x^{3} + x^{2} y + x\right) X^{22}  \\& + \left(x^{3} + x\right) X^{21} + \left(x^{8} + x^{7} + x^{4} + x^{3} + x^{2} + x + 1\right) X^{20} Y  \\& + \left(x^{9} + x^{8} y + x^{7} y + x^{7} + x^{5} + x^{4} y + x^{3} y + x^{3} + x^{2} y + x y + x + y + 1\right) X^{20} + \left(x^{6} + x^{4} + x^{2} + x\right) X^{19}  \\& + \left(x^{7} + x^{6} + x^{4} + x\right) X^{18} Y + \left(x^{10} + x^{9} + x^{7} y + x^{6} y + x^{6} + x^{5} + x^{4} y + x^{4} + x^{3} + x y\right) X^{18} \\&  + \left(x^{8} + x^{7} + x^{5} + x^{3} + x^{2} + x + 1\right) X^{17} + \left(x^{7} + x^{6} + x^{4} + x^{3}\right) X^{16} Y  \\& + \left(x^{11} + x^{10} + x^{7} y + x^{7} + x^{6} y + x^{4} y + x^{4} + x^{3} y + x^{3}\right) X^{16} + \left(x^{6} + x^{5} + x^{4} + x\right) X^{15} \\&+ \left(x^{7} + x^{3} + x^{2} + x\right) X^{14} Y \\&  + \left(x^{10} + x^{9} + x^{8} + x^{7} y + x^{7} + x^{6} + x^{5} + x^{4} + x^{3} y + x^{3} + x^{2} y + x y + x\right) X^{14}  \\& + \left(x^{9} + x^{8} + x^{6} + x^{5} + x^{4} + x^{2} + 1\right) X^{13} + \left(x^{8} + x^{5} + 1\right) X^{12} Y  \\& + \left(x^{11} + x^{10} + x^{9} + x^{8} y + x^{8} + x^{7} + x^{5} y + x^{5} + x^{4} + x^{3} + x + y + 1\right) X^{12} + \left(x^{9} + x^{8} + x^{4} + x^{3} + 1\right) X^{11} \\&  + \left(x^{8} + x^{6} + x^{4} + x + 1\right) X^{10} Y + \left(x^{9} + x^{8} y + x^{7} + x^{6} y + x^{5} + x^{4} y + x^{4} + x y + y\right) X^{10}  \\& + \left(x^{9} + x^{8} + x^{7} + x^{6} + x^{5} + x^{2} + x + 1\right) X^{9} + \left(x^{9} + x^{5} + x^{2} + x\right) X^{8} Y  \\& + \left(x^{11} + x^{9} y + x^{9} + x^{8} + x^{7} + x^{6} + x^{5} y + x^{2} y + x^{2} + x y\right) X^{8} + \left(x^{9} + x^{7} + x^{6} + x^{4} + x^{2}\right) X^{7}  \\& + \left(x^{8} + x^{5} + x^{4} + x^{3} + x^{2} + x + 1\right) X^{6} Y\\& + \left(x^{10} + x^{9} + x^{8} y + x^{8} + x^{5} y + x^{4} y + x^{4} + x^{3} y + x^{3} + x^{2} y + x y + y\right) X^{6} \\&  + \left(x^{9} + x^{6} + x^{5} + x^{3} + x^{2}\right) X^{5} + \left(x^{9} + x^{7} + x^{6} + x^{4} + x^{3}\right) X^{4} Y \\&  + \left(x^{11} + x^{10} + x^{9} y + x^{9} + x^{7} y + x^{6} y + x^{6} + x^{4} y + x^{3} y\right) X^{4} + \left(x^{9} + x^{6} + x^{5}\right) X^{3} + \left(x^{3} + x^{2}\right) X^{2} Y  \\& + \left(x^{11} + x^{10} + x^{8} + x^{5} + x^{3} y + x^{2} y\right) X^{2} + \left(x^{9} + x^{6} + x^{5}\right) Y + x^{12} + x^{9} y + x^{8} + x^{7} + x^{6} y + x^{5} y + x^{5}.
\end{align*}

Now, we calculate $F_{\le 12}$ such that $F_{\le 12}(d) = \ell_d^3 S_{\le d}(1,2)$ by using the identity recursion $F_{\le 12}^{(1)} - (g^{(1)})^3 F_{\le 12} = F_{12}^{(1)}$. Let $Z = (X^8+x)^3 (X^{16}+x+1)^3 F_{\le 12}$. The equation to be solved is
\begin{align*}
 (X^8+x)^3 (X^{16}+x+1)^3 (Z^{(1)} - N_{12}^{(1)}) = (N_g^{(1)})^3 Z.
\end{align*}
Put $Z = \sum _{i=0}^{48} a_iX^i + Y \sum_{j=0}^{40} b_jX^{j}$. This gives a system of $329$ equations in $90$ unknowns. The unique solution implies that $F_{\le 12} = N_{\le 12}/D_{\le 12}$, where $D _{\le 12} =  (x^2 + x) (X^8+x)^3 (X^{16}+x+1)^3)$, and $N_{\le 12}$ is 
\begin{align*}
&\left(x^{6} + x^{5} + x^{3} + x + 1\right) X^{48} + \left(x^{6} + x^{5} + x^{4} + x^{3}\right) X^{46} + \left(x^{8} + x^{7} + x^{4} + x^{3}\right) X^{44}  \\& + \left(x^{7} + x^{6} + x^{5} + x^{4} + x^{2} + x\right) X^{42}
+ \left(x^{6} + x^{5} + x^{4} + x^{3}\right) X^{41} + \left(x^{9} + x^{7} + x^{6} + x^{4} + x^{3} + x + 1\right) X^{40} \\&  + \left(x^{6} + x^{5} + x^{4} + x^{3}\right) X^{39} + \left(x^{7} + x^{6} + x^{5} + x\right) X^{38} 
+ \left(x^{7} + x^{3} + x^{2} + x\right) X^{37} + \left(x^{6} + x^{5} + x^{4} + x^{3}\right) X^{36} Y    \\&   + \left(x^{8} + x^{7} + x^{6} y + x^{5} y + x^{5} + x^{4} y + x^{3} y + x\right) X^{36}
+ \left(x^{7} + x^{3} + x^{2} + x\right) X^{35} \\&+ \left(x^{7} + x^{6} + x^{4} + x^{3} + x^{2} + x\right) X^{34}   + \left(x^{6} + x^{5} + x^{3} + x^{2}\right) X^{33} + \left(x^{7} + x^{3} + x^{2} + x\right) X^{32} Y\\& 
+ \left(x^{9} + x^{7} y + x^{7} + x^{6} + x^{5} + x^{4} + x^{3} y + x^{2} y + x y + x + 1\right) X^{32} + \left(x^{6} + x^{5} + x^{3} + x^{2}\right) X^{31} + \left(x^{8} + x\right) X^{30}\\& 
+ \left(x^{7} + x^{5} + x^{4} + x\right) X^{29} + \left(x^{6} + x^{5} + x^{3} + x^{2}\right) X^{28} Y  \\& + \left(x^{10} + x^{9} + x^{8} + x^{6} y + x^{5} y + x^{3} y + x^{3} + x^{2} y + x^{2} + x + 1\right) X^{28} 
+ \left(x^{7} + x^{5} + x^{4} + x\right) X^{27}  \\& + \left(x^{9} + x^{8} + x^{6} + x^{5} + x^{4} + x^{3}\right) X^{26} + \left(x^{8} + x^{7} + x^{5} + x^{4}\right) X^{25} + \left(x^{7} + x^{5} + x^{4} + x\right) X^{24} Y\\& 
+ \left(x^{11} + x^{10} + x^{9} + x^{7} y + x^{5} y + x^{4} y + x^{3} + x^{2} + x y\right) X^{24} + \left(x^{8} + x^{7} + x^{5} + x^{4}\right) X^{23}  \\& + \left(x^{9} + x^{6} + x^{5} + x^{4} + x^{3} + x\right) X^{22}
+ \left(x^{9} + x^{7} + x^{6} + x^{3}\right) X^{21} + \left(x^{8} + x^{7} + x^{5} + x^{4}\right) X^{20} Y \\&
+ \left(x^{9} + x^{8} y + x^{8} + x^{7} y + x^{7} + x^{6} + x^{5} y + x^{5} + x^{4} y + x^{3} + x^{2} + 1\right) X^{20} + \left(x^{9} + x^{7} + x^{6} + x^{3}\right) X^{19}\\&
+ \left(x^{8} + x^{6} + x^{5} + x^{3} + x^{2} + x + 1\right) X^{18} + \left(x^{7} + x^{5} + x^{4} + x\right) X^{17} + \left(x^{9} + x^{7} + x^{6} + x^{3}\right) X^{16} Y \\&
+ \left(x^{11} + x^{10} + x^{9} y + x^{8} + x^{7} y + x^{6} y + x^{4} + x^{3} y + x^{2} + x\right) X^{16} + \left(x^{7} + x^{5} + x^{4} + x\right) X^{15} \\&
+ \left(x^{9} + x^{7} + x^{3} + x + 1\right) X^{14} + \left(x^{8} + x^{7} + x^{6} + x^{4} + x^{3} + x^{2} + 1\right) X^{13} + \left(x^{7} + x^{5} + x^{4} + x\right) X^{12} Y \\&
+ \left(x^{11} + x^{10} + x^{9} + x^{7} y + x^{7} + x^{5} y + x^{4} y + x^{3} + x y\right) X^{12} + \left(x^{8} + x^{7} + x^{6} + x^{4} + x^{3} + x^{2} + 1\right) X^{11} \\&
+ \left(x^{10} + x^{9} + x^{7} + x^{4} + x^{3} + x^{2} + x\right) X^{10} + \left(x^{9} + x^{8} + x^{6} + x^{4} + x^{2} + x\right) X^{9}  \\& + \left(x^{8} + x^{7} + x^{6} + x^{4} + x^{3} + x^{2} + 1\right) X^{8} Y \\& 
+ \left(x^{12} + x^{11} + x^{9} + x^{8} y + x^{8} + x^{7} y + x^{6} y + x^{6} + x^{5} + x^{4} y + x^{4} + x^{3} y + x^{3} + x^{2} y + y\right) X^{8} \\&
+ \left(x^{9} + x^{8} + x^{6} + x^{4} + x^{2} + x\right) X^{7} + \left(x^{10} + x^{8} + x^{7} + x^{6} + x^{3}\right) X^{6} + \left(x^{10} + x^{8} + x^{7} + x^{6} + x^{3}\right) X^{5} \\&
+ \left(x^{9} + x^{8} + x^{6} + x^{4} + x^{2} + x\right) X^{4} Y + \left(x^{9} y + x^{9} + x^{8} y + x^{8} + x^{6} y + x^{5} + x^{4} y + x^{4} + x^{2} y + x y\right) X^{4} \\&
+ \left(x^{10} + x^{8} + x^{7} + x^{6} + x^{3}\right) X^{3} + \left(x^{10} + x^{8} + x^{7} + x^{6} + x^{3}\right) Y\\& + x^{12} + x^{11} + x^{10} y + x^{10} + x^{8} y + x^{7} y + x^{7} + x^{6} y + x^{3} y. 
\end{align*}

We first calculate  $F_3$, such that $F_3(d) = \ell_d^3 S_d(3)$. Since
\begin{align*}
 F_{<1}(d) = \ell_d S_{<d}(1) = \frac{ \ell_d^2 A_{d1} }{\ell_d A_{d0}},
\end{align*}
it follows that $\ell_d^2 A_{d1} = (\ell_d A_{d0}) ( \ell_d S_{<d}(1)) = F_1(d)F_{<1}(d)$. 
\begin{align*}
 \ell_d^3 S_d(3) &= \ell_d A_{d0} ( \ell_d^2 A_{d1} + (\ell_d A_{d0})^2)
  = F_1(d) ( F_1(d) F_{<1}(d) + F_1(d)^2 ).
\end{align*}
We then define $F_3 = F_1^2 F_{<1}+ F_1^3$.

Finally, $F_{<3}$ is calculated such that $F_{<3}(d) = \ell_d^3 S_{<d}(3)$. In a similar way as $F_{<1}$ was obtained, we obtain the equation 
\begin{align*}
F_{<3}^{(1)} = (g^{(1)} )^3 ( F_{<3} + F_3). 
\end{align*}
Making the change of variable $Z = (X^8+x)^3 ( X^{16}+x+1)^3 F_{<3}$ we obtain $(X^8+x)^3 ( X^{16}+x+1)^3 Z = (N_g^{(1)})^3 (Z+N_3)$. We put $Z = \sum _{i=0}^{48} a_i X^i + Y \sum _{j=0} ^{45}b_j X^j$ and obtaing a system of $334$ equations with $95$ unknowns. Then $F_{<3} = N_{<3}/D_{<3}$, where $D_{<3} = ( (x^2+x) (X^8 + x)^3  (X^{16} + x + 1)^3$ and $N_{<3}$ is 
\begin{align*}
&\left(x^{8} + x^{6} + x^{5} + x^{4} + x^{3} + x + 1\right) X^{48} + \left(x^{4} + x^{2}\right) X^{46} + \left(x^{8} + x^{6} + x^{4} + x\right) X^{44}\\& + \left(x^{5} + x^{3} + x^{2} + x\right) X^{42}
+ \left(x^{4} + x^{2}\right) X^{41} + \left(x^{10} + x^{7} + x^{3} + x^{2} + 1\right) X^{40} + \left(x^{4} + x^{2}\right) X^{39} \\&+ \left(x^{8} + x^{6} + x^{2} + x\right) X^{38} + \left(x^{5} + x^{4} + x^{3} + x\right) X^{37} 
+ \left(x^{4} + x^{2}\right) X^{36} Y\\& + \left(x^{10} + x^{9} + x^{8} + x^{7} + x^{5} + x^{4} y + x^{4} + x^{2} y\right) X^{36} + \left(x^{5} + x^{4} + x^{3} + x\right) X^{35} \\&
+ \left(x^{9} + x^{8} + x^{7} + x^{6} + x^{5} + x^{3} + x^{2} + x\right) X^{34} + \left(x^{8} + x^{6} + x^{5} + x^{4} + x^{3} + x^{2}\right) X^{33} \\&+ \left(x^{5} + x^{4} + x^{3} + x\right) X^{32} Y
+ \left(x^{8} + x^{7} + x^{6} + x^{5} y + x^{4} y + x^{4} + x^{3} y + x^{3} + x^{2} + x y + 1\right) X^{32} \\&+ \left(x^{8} + x^{6} + x^{5} + x^{4} + x^{3} + x^{2}\right) X^{31} 
+ \left(x^{4} + x\right) X^{30} + \left(x^{9} + x^{7} + x^{4} + x\right) X^{29}\\& + \left(x^{8} + x^{6} + x^{5} + x^{4} + x^{3} + x^{2}\right) X^{28} Y \\&
+ \left(x^{11} + x^{8} y + x^{7} + x^{6} y + x^{6} + x^{5} y + x^{5} + x^{4} y + x^{4} + x^{3} y + x^{3} + x^{2} y + 1\right) X^{28} \\& + \left(x^{9} + x^{7} + x^{4} + x\right) X^{27} 
+ \left(x^{10} + x^{9} + x^{8} + x^{5} + x^{3} + x^{2}\right) X^{26} + \left(x^{9} + x^{7}\right) X^{25} \\& + \left(x^{9} + x^{7} + x^{4} + x\right) X^{24} Y\\&
+ \left(x^{12} + x^{9} y + x^{9} + x^{8} + x^{7} y + x^{7} + x^{6} + x^{5} + x^{4} y + x^{4} + x^{3} + x^{2} + x y\right) X^{24} + \left(x^{9} + x^{7}\right) X^{23} \\&
+ \left(x^{10} + x^{9} + x^{5} + x^{4} + x^{3} + x\right) X^{22} + \left(x^{10} + x^{8} + x^{7} + x^{5} + x^{3} + x^{2}\right) X^{21} + \left(x^{9} + x^{7}\right) X^{20} Y \\&
+ \left(x^{11} + x^{9} y + x^{9} + x^{7} y + x^{5} + x^{4} + x + 1\right) X^{20} + \left(x^{10} + x^{8} + x^{7} + x^{5} + x^{3} + x^{2}\right) X^{19} \\&+ \left(x^{10} + x + 1\right) X^{18} 
+ \left(x^{9} + x^{8} + x^{7} + x^{4} + x^{2} + x\right) X^{17} + \left(x^{10} + x^{8} + x^{7} + x^{5} + x^{3} + x^{2}\right) X^{16} Y \\&
+ \left(x^{10} y + x^{8} y + x^{7} y + x^{7} + x^{5} y + x^{5} + x^{4} + x^{3} y + x^{3} + x^{2} y\right) X^{16} + \left(x^{9} + x^{8} + x^{7} + x^{4} + x^{2} + x\right) X^{15} \\&
+ \left(x^{6} + x^{5} + x^{4} + x^{2} + 1\right) X^{14} + \left(x^{10} + x^{9} + x^{8} + x^{7} + x^{4} + x^{2} + 1\right) X^{13} \\&+ \left(x^{9} + x^{8} + x^{7} + x^{4} + x^{2} + x\right) X^{12} Y \\&
+ \left(x^{12} + x^{11} + x^{10} + x^{9} y + x^{8} y + x^{8} + x^{7} y + x^{7} + x^{5} + x^{4} y + x^{4} + x^{2} y + x^{2} + x y + x\right) X^{12}\\&
+ \left(x^{10} + x^{9} + x^{8} + x^{7} + x^{4} + x^{2} + 1\right) X^{11} + \left(x^{11} + x^{8} + x^{7} + x^{2} + x\right) X^{10} \\&+ \left(x^{10} + x^{9} + x^{8} + x^{7} + x^{6} + x^{5}\right) X^{9} 
+ \left(x^{10} + x^{9} + x^{8} + x^{7} + x^{4} + x^{2} + 1\right) X^{8} Y \\&
+ \left(x^{13} + x^{10} y + x^{9} y + x^{9} + x^{8} y + x^{8} + x^{7} y + x^{7} + x^{6} + x^{4} y + x^{3} + x^{2} y + y\right) X^{8} \\&
+ \left(x^{10} + x^{9} + x^{8} + x^{7} + x^{6} + x^{5}\right) X^{7} + \left(x^{11} + x^{10} + x^{9} + x^{6} + x^{5} + x^{2} + x\right) X^{6} \\&
+ \left(x^{11} + x^{10} + x^{9} + x^{6} + x^{5} + x^{2} + x\right) X^{5} + \left(x^{10} + x^{9} + x^{8} + x^{7} + x^{6} + x^{5}\right) X^{4} Y\\&
+ \left(x^{10} y + x^{10} + x^{9} y + x^{9} + x^{8} y + x^{8} + x^{7} y + x^{7} + x^{6} y + x^{5} y\right) X^{4} \\&+ \left(x^{11} + x^{10} + x^{9} + x^{6} + x^{5} + x^{2} + x\right) X^{3} + \left(x^{11} + x^{10} + x^{9} + x^{6} + x^{5} + x^{2} + x\right) Y\\& + x^{11} y + x^{11} + x^{10} y + x^{9} y + x^{9} + x^{7} + x^{6} y + x^{6} + x^{5} y + x^{5} + x^{4} + x^{2} y + x y. 
\end{align*}

Notice that denominators of $F_{\le 12}$ and $F_{ <3}$ match. 
The leading term of $(x^8 + x^6 + x^5 + x^4 + x^3 + x + 1) N_{\le 12} - (x^6 + x^5 + x^3 + x + 1) N_{<3}$ is $(x^{14} + x^{13} + x^{11} + x^{9} + x^7 + x^6 + x^4 + x^2)X^{46} $, so that the degree of $E(d) = (x^8 + x^6 + x^5 + x^4 + x^3 + x + 1)F_{\le 12}(d)- (x^6 + x^5 + x^3 + x + 1) F_{<3}(d)$ for $d\ge 3$ is $- (-24+52 \times 2^{d-2})$;  since the degree of $\ell_d$ is  $-(2^{d+2} -2)$ the degree of $E(d)/\ell_d^3$  is $-(-12+2^d)$ ($d \ge 2$);  therefore, $(x^8 + x^6 + x^5 + x^4 + x^3 + x + 1)S_{\le d}(1,2)- (x^6 + x^5 + x^3 + x + 1) S_{<3}(d)$  tends to zero as $d$ tends to infinity. 
\end{proof}

\end{document}